\documentclass[12pt,twoside]{amsart}
\input{includeNice3}
\usepackage{mabliautoref}
\usepackage{amssymb}
\usepackage{amscd}
\usepackage[abbrev,alphabetic]{amsrefs}
\usepackage{hyperref, cleveref}

\usepackage[OT2,T1]{fontenc}
\DeclareSymbolFont{cyrletters}{OT2}{wncyr}{m}{n}
\DeclareMathSymbol{\Sha}{\mathalpha}{cyrletters}{"58}

\newcommand{\Bref}[1]{(\hyperref[item:Brauer]{\ensuremath{B_{#1}}})}
\newcommand{\PIDref}[1]{(\hyperref[item:PID]{\ensuremath{PID_{#1}}})}
\newcommand{\AFref}[1]{(\hyperref[item:AbrashkinFontaine]{\ensuremath{AV_{#1}}})}
\newcommand{\Hodgeref}[1]{(\hyperref[item:HodgeNumbers]{\ensuremath{h_{#1}^{p,q}}})}

\usepackage[a4paper,margin=1in]{geometry}  

\usepackage{xparse}
\let\realItem\item % save a copy of the original item
\makeatletter
\NewDocumentCommand\myItem{ o }{%
   \IfNoValueTF{#1}%
      {\realItem}% add an item
      {\realItem[#1]\def\@currentlabel{#1}}% add an item and update label
}
\usepackage{enumitem}
\setlist[enumerate]{
    before=\let\item\myItem,       % use \myItem in enumerate
    label=\textnormal{(\arabic*)}, % format the label
    widest=(2')                    % set the widest label
}

\usepackage{xypic}
\usepackage{tikz-cd}
\usepackage{caption}

\usepackage{url}
\RequirePackage{booktabs, multirow}
\RequirePackage{pgf}
\title[]
{Enriques' characterization of Abelian Surfaces in positive characteristic} 
\author{Jefferson Baudin, Gebhard Martin} 

\address{Mathematisches Institut  Universität Bonn, Endenicher Allee 60, 53115 Bonn, Germany}
\email{gmartin@math.uni-bonn.de} 

\address{EPFL SB MATH CAG
	MA C3 615 (B\^atiment MA)
	Station 8
	CH-1015 Lausanne}
\email{jefferson.baudin@epfl.ch}

\begin{document}

\maketitle

\begin{abstract}
Extending Enriques' characterization to algebraically closed fields of characteristic $p \geq 7,$ we show that every smooth projective surface $X$ with $h^1(X,\mathcal{O}_X) = 2$ and $p_1(X) = p_2(X) = 1$ is birational to an Abelian surface.
This characterization fails if $p \leq 5$ and we give a sharp alternative.
\end{abstract}

\tableofcontents

\section{Introduction}
In \cite{Enriques}, Enriques characterized Abelian surfaces birationally among smooth projective complex surfaces $X$ as those with irregularity $h^1(X,\mathcal{O}_X) = 2$ and first and second plurigenus $p_1(X) = p_2(X) = 1$. This has been generalized to higher dimensions by Chen--Hacon \cite{Chen_Hacon_Characterization_of_abelian_varieties} (see also \cite{Jiang_An_effective_version_of_a_thm_of_Kawamata_on_the_Albanese_map, Pareschi_Basic_results_on_irr_vars_via_FM_methods}).

First steps towards such a characterization of Abelian surfaces in positive characteristics have been made by Ferrari \cite{Ferrari} (see also \cite{Baudin_Effective_Characterization_of_ordinary_abelian_varieties ,Hacon_Patakfalvi_Zhang_Bir_char_of_AVs} for statements in higher dimensions). The problem is particularly subtle in characteristic $2$ and $3$, because some \mbox{(quasi-)}bielliptic surfaces with trivial canonical bundle cannot be distinguished from Abelian surfaces using only Hodge numbers and plurigenera. Thus, the following formulation of our main theorem is natural:

\begin{theorem}\label{thm: mainintro}
Let $X$ be a smooth projective surface over an algebraically closed field of characteristic $p > 0$. Then, $X$ is birational to an Abelian or (quasi-)bielliptic surface $S$ with $\omega_S \cong \mathcal{O}_S$ if and only if $h^1(X,\mathcal{O}_X) = 2$ and $p_i(X) = 1$ for
$$
i \leq \begin{cases}
2  \text{ if } p \geq 7, \\
3 \text{ if }  p = 5, \\
4 \text{ if }  p \leq 3.
\end{cases}
$$
\end{theorem}

To obtain a characterization of Abelian surfaces, one can then add another invariant that distinguishes them from (quasi-)bielliptic surfaces with trivial canonical. 
For example, assuming $p \geq 5$, these surfaces do not exist, so that we get the following corollary:

\begin{corollary} \label{cor: mainintro}
Let $X$ be a smooth projective surface over an algebraically closed field of characteristic $p \geq 5$. Then, $X$ is birational to an Abelian surface if and only if $h^1(X,\mathcal{O}_X) = 2$ and $p_i(X) = 1$ for
$$
i \leq \begin{cases}
2  \text{ if } p \geq 7. \\
3 \text{ if }  p = 5. \\
\end{cases}
$$
\end{corollary}

Alternatively, we have the following easier characterization if we replace $h^1(X,\mathcal{O}_X)$ by $\frac{1}{2}b_1(X)$, the dimension of the Picard scheme of $X$. Note that these two quantities will differ if and only if the Picard scheme is non-reduced, a phenomenon that can only happen in positive characteristic:

\begin{proposition} \label{prop: mainintro2}
Let $X$ be a smooth projective surface over an algebraically closed field $k$ of characteristic $p > 0$. Then, $X$ is birational to an Abelian surface if and only if $b_1(X) = 4$ and $p_1(X) = p_2(X) = 1$.
\end{proposition}

These characterizations are sharp in the sense that no assumption on the higher plurigenera can be dropped. This follows from the three examples we give in \autoref{sec: examples}. \\

\noindent 
\textbf{Acknowledgements.}
The authors would like to thank Sofia Tirabassi for discussions related to the topic of this article. JB was funded by grant \#20021/231484 from the Swiss National Science Foundation.
% \begin{theorem}
% Let $X$ be a smooth projective surface with $b_1(X) = 4$ and $p_2(X) = 1$. Then, $X$ is an Abelian surface.
% \end{theorem}

% This is a consequence of the following:
\section{Preliminaries}
Let us fix our setup and notation:

\begin{itemize}
\item Throughout, we work over an algebraically closed field $k$ of characteristic $p \geq 0$.
\item Unless stated otherwise, $X$ denotes a smooth projective surface over $k$.
\item We define the $i$-th plurigenus as $p_i(X) \coloneqq h^0(X,\omega_X^{\otimes i})$.
\item The Kodaira dimension of $X$ is denoted by $\kappa(X)$.
\item We say that $X$ has maximal Albanese dimension if it admits a generically finite morphism to an Abelian variety.
\end{itemize}
We assume that the reader is familiar with the classification of minimal surfaces of special type (see \cite{BombieriMumfordI, BombieriMumford, BombieriMumfordIII}).
We will use the following well-known consequence of the fact that $h^1(\mathbb{P}^1,\mathcal{O}_{\mathbb{P}^1}) = 0$:

\begin{lemma} \label{lem: contract}
Let $f \colon X \to A$ be a morphism from a rationally chain connected variety $X$ to an Abelian variety $A$. Then, $f(X)$ is a point.
\end{lemma}

Recall the Noether formula
\begin{equation} \label{eq: Noether} \tag{N}
\chi(\mathcal{O}_X) = \frac{K_X^2 + c_2(X)}{12}
\end{equation}
and the canonical bundle formula for genus $1$ fibrations $g: X \to B$:
\begin{equation} \label{eq: canonical} \tag{C}
\omega_X \cong g^*(\mathcal{L}^{-1} \otimes_{\mathcal{O}_B} \omega_B) \otimes_{\mathcal{O}_X} \mathcal{O}_X(\sum_{b \in B} a_b F_b).
\end{equation}
Here:
\begin{itemize}
\item  $R^1 g_* \mathcal{O}_X \cong \mathcal{L} \oplus \mathcal{T}$ with $\mathcal{L}$ invertible and $\mathcal{T}$ torsion;
\item For $b \in B$, $m_bF_b$ is a fiber of $g$ and $m_b$ is its multiplicity;
\item $0 \leq a_b < m_b$;
\item $a_b = m_b-1$ for $b \not \in {\rm Supp}(\mathcal{T})$;
\item ${\rm deg}(\mathcal{L}^{-1} \otimes_{\mathcal{O}_B} \omega_B) = 2g(B) - 2 + \chi(\mathcal{O}_X) + {\rm length}(\mathcal{T})$.
\end{itemize} 
We also need the following characterization of wild fibers of genus $1$ fibrations \cite[Proposition 4]{BombieriMumford}:

\begin{lemma} \label{lem: wild}
Let $g \colon X \to B$ be a genus $1$ fibration. For a multiple fiber $mF$ of $g$, the following are equivalent:
\begin{enumerate}
\item $\mathcal{T}_{g(F)} \neq 0$;
\item $h^0(mF,\mathcal{O}_{mF}) \geq 2$;
\item ${\rm ord}(\mathcal{O}_F(F)) < m$.
\end{enumerate}
\end{lemma}

\section{Surfaces of maximal Albanese dimension}
In this section, we prove a characteristic-free version of \cite[Proposition 4.5]{Kollar}:

\begin{theorem} \label{thm: main}
Let $X$ be a surface of maximal Albanese dimension. Then, either $X$ is birational to an Abelian surface or $p_2(X) \geq 2$.
\end{theorem} 

\begin{proof}
By rigidity, we may assume that $X$ is minimal.
As $X$ is of maximal Albanese dimension, it is not birationally uniruled by \autoref{lem: contract}, hence $\kappa(X) \geq 0$ (one could also deduce this from \cite[Theorem C]{Baudin_Positive_characteristic_generic_vanishing_theory}).

Now, seeking a contradiction, assume that $X$ is not an Abelian surface and $p_2(X) \leq 1$. By \cite[Corollary 1.8]{Ekedahl_Canonical_models_of_surfaces_of_general_type_in_positive_characteristic}, $X$ is not of general type. Note also that $\kappa(X) \neq 0$ by either classification of surfaces \cite{BombieriMumford} or by \cite[Theorem 0.2]{Hacon_Patakfalvi_Zhang_Bir_char_of_AVs}. We therefore have  $\kappa(X) = 1$, so $X$ admits a canonical genus $1$ fibration $g\colon X \to B$.
As $X$ has maximal Albanese dimension, the general fiber of $g$ is not rational by \autoref{lem: contract}, so $g$ is generically smooth.

In fact, again because $X$ has maximal Albanese dimension, no fiber of $g$ is contracted by $f$. As the underlying reduced scheme of a fiber of $g$ is either smooth or a union of rational curves, \autoref{lem: contract} implies that it is smooth. Thus, by \cite[Theorem 6.6]{LiuLorenziniRaynaud}, the associated Jacobian fibration $J(g)\colon J(X) \to B$ is smooth and hence $\chi(\mathcal{O}_{J(X)}) = 0$. By \cite[Theorem 3.1]{LiuLorenziniRaynaud} and the Leray spectral sequence for $\mathcal{O}_X$ with respect to $g$, we have $\chi(\mathcal{O}_{J(X)}) = \chi(\mathcal{O}_X)$. As $p_1(X) \leq 1$ and $h^1(X,\mathcal{O}_X) \geq 2$, we deduce that $h^1(X,\mathcal{O}_X) = 2$, $p_1(X) = 1$, $b_1(X) = 4$ and $b_2(X) = 6$. By \cite[Lemma 3.4]{KatsuraUeno}, $b_1(X) = 4$ and the fact that the Albanese map of $X$ does not contract fibers of $g$ implies that $g(B) = 1$.

Now, consider the canonical bundle formula \eqref{eq: canonical}.
Since $p_2(X) \leq 1$, the tensor square of the right-hand side cannot contain the pullback of a line bundle of degree at least $2$ on $B$. This implies that $\mathcal{T}$ is trivial and that there is at most one multiple fiber $F$, say of multiplicity $m$, and the coefficient $a$ in the canonical bundle formula is $m-1$. If $m = 1$, then $\omega_X \cong g^*(\mathcal{L}^{-1})$ for $\mathcal{L}$ of degree $1$, so $K_X \equiv 0$ and $X$ is an Abelian surface, a contradiction.

Thus, assume that $m > 1$ and write $\omega_X \cong g^* \mathcal{L}^{-1} \otimes_{\mathcal{O}_X} \mathcal{O}_X((m-1)F)$ with ${\rm deg}(\mathcal{L}) = 0$. The inclusion $\iota \colon  F \hookrightarrow X$ induces a morphism of identity components of Picard schemes $\iota^* \colon {\rm Pic}_{X/k}^0 \to {\rm Pic}^0_{F/k}$. We claim that $\iota^*$ is surjective. As $\iota^*$ is a morphism of Abelian varieties and the target is an elliptic curve, it suffices to show that it has non-zero differential. This differential is induced by the long exact sequence in cohomology associated to
\begin{equation}\label{equation_exact_seq}
    0 \to \mathcal{O}_X(-F) \to \mathcal{O}_X \to \mathcal{O}_F \to 0.
\end{equation}
By Serre duality and the projection formula, 
$$
h^2(X,\mathcal{O}_X(-F) ) = h^0(X,g^* \mathcal{L}^{-1} \otimes \mathcal{O}_X(mF)) = h^0(B, \mathcal{L}^{-1}(g(F))) = 1,
$$
because ${\rm deg}(\mathcal{L}^{-1}(g(F))) = 1$. As $h^2(X,\mathcal{O}_X) = 1$, this implies that $H^1(X,\mathcal{O}_X) \to H^1(F,\mathcal{O}_F)$ is surjective, as desired.

By \autoref{lem: wild}, $\mathcal{O}_F(-F)$ corresponds to a point of exact order $m$ on ${\rm Pic}_{F/k}^0$.  We thus find an invertible sheaf $\mathcal{M} \in {\rm Pic}^0_{X/k}$ with $\mathcal{M}|_{F} \cong \mathcal{O}_F(-F)$. 
We now tensor \autoref{equation_exact_seq} with $\mathcal{M}(F)$:
\begin{equation} \label{eq: ses2}
0 \to \mathcal{M} \to \mathcal{M}(F) \to \mathcal{O}_F \to 0.
\end{equation}
It remains to show that all of the following three cases lead to a contradiction:
\begin{enumerate}
\item $h^0(X,\mathcal{M}(F)) \neq 0$.
\item $h^2(X,\mathcal{M}(F)) \neq 0$.
\item $h^0(X,\mathcal{M}(F)) = h^2(X,\mathcal{M}(F)) = 0$
\end{enumerate}

Case $(1)$: If $h^0(X,\mathcal{M}(F)) \neq 0$, pick an effective divisor $D$ with $\mathcal{O}_X(D) \cong \mathcal{M}(F)$. Then, $D.F = {\rm deg}(\mathcal{M}(F)|_F) = 0$, so $D$ is supported on fibers of $g$. Write $D = \sum_{b \in B} c_bF_b + cF$ with $c,c_b \geq 0$. Then,
$
\mathcal{M} \cong \mathcal{O}_X(\sum_{b \in B} c_bF_b + (c-1)F),
$
and hence
$$
\mathcal{O}_F(-F) \cong \mathcal{M}|_F \cong \mathcal{O}_F((c-1)F).
$$
Thus, $m \mid c$. On the other hand, since $\mathcal{M} \in {\rm Pic}^0_{X/k}$, its restriction to every curve on $X$ has degree $0$. In particular, we can take an ample divisor $H$ on $X$ and get
$$
0 = {\rm deg}_H(\mathcal{M}|_H) = \left( m\sum_{b \in B} c_b + c-1\right) H.F,
$$
so that $$m \sum_{b \in B} c_b + c-1 = 0.$$
Since $m > 1$ and $m \mid c$, the left-hand side is not divisible by $m$, a contradiction.
\smallskip

Case $(2)$: If $h^2(X,\mathcal{M}(F)) \neq 0$, then $h^0(X,\mathcal{M}^{-1}((m-2)F) \otimes g^*\mathcal{L}^{-1}) \neq 0$ by Serre duality. Pick an effective divisor $D$ with $\mathcal{O}_X(D) \cong \mathcal{M}^{-1}((m-2)F) \otimes g^*\mathcal{L}^{-1}$. As before, $D.F = 0$, so $D$ is supported on fibers and we can write $D \sim  \sum_{b \in B} c_bF_b + cF$. Then,
$$
\mathcal{M} \cong g^* \mathcal{L}^{-1} \left(-\sum_{b \in B}c_bF_b + (m-2-c)F\right),
$$
and hence
$$
\mathcal{O}_F(-F) \cong \mathcal{M}|_F \cong \mathcal{O}_F((m-2-c)F).
$$
Thus, $m \mid (c+1)$. Intersecting with an ample divisor as in the previous case, we obtain
$$
m\sum_{b \in B} c_b +  m-2-c = 0.
$$
Since $m > 1$ and $m \mid (c+1)$, the left-hand side is not divisible by $m$, a contradiction.
\smallskip

Case $(3)$: If $h^0(X,\mathcal{M}(F)) = h^2(X,\mathcal{M}(F)) = 0$, we use that $\chi(\mathcal{M}) = \chi(\mathcal{O}_F) = 0$, so that the short exact sequence \eqref{eq: ses2} shows that also $h^1(X,\mathcal{M}(F)) = 0$. But then, by \eqref{eq: ses2}, we have $h^0(X,\mathcal{M}^{-1}((m-1)F) \otimes g^*\mathcal{L}^{-1}) =h^2(X,\mathcal{M}) =1$. Pick an effective divisor $D$ with $\mathcal{O}_X(D) \cong \mathcal{M}^{-1}((m-1)F) \otimes g^*\mathcal{L}^{-1}$. Similarly to the previous cases, we write $D \sim \sum_{b \in B} c_bF_b + cF$ and we obtain $m \mid c$ and
$$
m\sum_{b \in B} c_b + m-c-1 = 0,
$$
thus leading to a final contradiction in the case where $m > 1$. This finishes the proof. 
\end{proof}

\begin{remark}
Note that in the above proof, in order to show that there does not exist an elliptic surface $g \colon X \to B$ with $g(B) =1$ and exactly one tame multiple fiber, we use in an an essential way the existence of an ample divisor on $X$, or equivalently, the fact that $X$ is an algebraic surface. In fact, one can produce non-algebraic (complex analytic or rigid analytic) elliptic surfaces with these invariants using logarithmic transformations \cite{Kodaira,Mitsui}.
\end{remark}

\section{Surfaces with $p_2(X) = 1$}
We begin with the following observation on the coherent Euler characteristic of surfaces with $p_2(X) = 1$.

\begin{proposition} \label{lem: chi1}
Let $X$ be a smooth projective surface with $p_2(X) = 1$. Then,
$$
0 \leq \chi(\mathcal{O}_X) \leq 2.
$$
\end{proposition}
\begin{proof}
The equality $p_2(X) = 1$ implies $\kappa(X) \geq 0$ and $p_1(X) \leq 1$. In particular, $\chi(\mathcal{O}_X) \leq 2$ is clear and we may assume without loss of generality that $X$ is minimal.

If $\kappa(X) = 2$, the claim holds by \cite[Theorem 1.5]{slope} and if $\kappa(X) = 0$, the claim holds by classification. 
Thus, we may assume that $\kappa(X) = 1$ and consider the Iitaka fibration $g\colon X \to B$ and the Albanese map $a\colon X \to A$.

If no fiber of $g$ is contracted by $a$, then, as in the proof of \autoref{thm: main}, the Jacobian $J(g)\colon J(X) \to B$ of $g$ is smooth, so $0 = \chi(\mathcal{O}_{J(X)}) = \chi(\mathcal{O}_X)$.

If some fiber of $g$ is contracted by $a$, then $g(B) = \frac{1}{2}b_1(X)$ by \cite[Lemma 3.4]{KatsuraUeno}. If in addition $b_1(X) = 0$, then $c_2(X) > 0$, so $\chi(\mathcal{O}_X) > 0$ by \eqref{eq: Noether}. We may therefore assume that $b_1(X) > 0$. Write $R^1g_* \mathcal{O}_X = \mathcal{L} \oplus \mathcal{T}$ with $\mathcal{L}$ invertible and $\mathcal{T}$ torsion. By \eqref{eq: canonical} and Riemann--Roch on the base curve $B$, we have
\begin{eqnarray*}
\chi((\mathcal{L}^{-1} \otimes \omega_B)^{\otimes 2}) & = & 3g(B) - 3 + 2\chi(\mathcal{O}_X) + 2{\rm length}(\mathcal{T}) \\
&=&\frac{3}{2}b_1(X) - 2h^1(X,\mathcal{O}_X) + 2p_1(X) + 2{\rm length}(\mathcal{T}) -1
\end{eqnarray*}
Recall that $H^2(X,\mathcal{O}_X)$ is an obstruction space for deformations of line bundles on $X$, so that
 $$
h^1(X,\mathcal{O}_X) - p_1(X)\leq \frac{1}{2}b_1(X).
$$
Thus, we have
$$
\chi((\mathcal{L}^{-1} \otimes \omega_B)^{\otimes 2})  \geq \frac{1}{2}b_1(X) + 2 {\rm length}(\mathcal{T}) - 1.
$$
Again by \eqref{eq: canonical}, we have
\begin{equation} \label{eq: p2ineq}
p_2(X) \geq h^0(B,(\mathcal{L}^{-1} \otimes \omega_B)^{\otimes 2}) \geq \chi((\mathcal{L}^{-1} \otimes \omega_B)^{\otimes 2}) \geq \frac{1}{2}b_1(X) + 2 {\rm length}(\mathcal{T}) - 1
\end{equation}
Now, \eqref{eq: Noether} yields the inequality $b_1(X) \geq 1 - 6 \chi(\mathcal{O}_X)$ and thus
\begin{equation} \label{eq: nowild}
 p_2(X) \geq  2 {\rm length}(\mathcal{T}) - 3 \chi(\mathcal{O}_X) - \frac{1}{2}.
\end{equation}
Hence, $p_2(X) = 1$ implies $\chi(\mathcal{O}_X) \geq 0$.
\end{proof}

Enriques' characterization theorem is about singling out Abelian surfaces among minimal surfaces with $\chi(\mathcal{O}_X) = 0$ and $p_2(X) = 1$, so let us finish this section by studying this case in detail.

\begin{proposition} \label{lem: chi2}
Let $X$ be a minimal smooth projective surface with $p_2(X) = 1$ and $\chi(\mathcal{O}_X) = 0$. Then, one of the following holds:
\begin{enumerate}
\item $b_1(X) = 4$ and $X$ is an Abelian surface.
\item $b_1(X) = 2$ and there is a genus $1$ fibration $g \colon X \to B$ with $g(B) \in \{0,1\}$. In this case:
\begin{itemize}
\item If $g(B) = 0$, then the Jacobian of $g$ is smooth.
\item If $g(B) = 1$, then $g$ is the Albanese map of $X$ and admits no wild fibers.
\end{itemize}
\end{enumerate}
\end{proposition}
\begin{proof}
If $\kappa(X) = 0$, the claim follows from classification.
So, let $g \colon X \to B$ be the Iitaka fibration and $a \colon X \to A$ the Albanese map.

If no fiber of $g$ is contracted by $a$, then $g(B) = \frac{1}{2}b_1(X) - 1$ by \cite[Lemma 3.4]{KatsuraUeno} and the Jacobian of $g$ is a smooth morphism. As $p_1(X) \leq 1$, we conclude 
$$
2 = 2h^1(X,\mathcal{O}_X) - 2p_1(X) \leq b_1(X) \leq 2h^1(X,\mathcal{O}_X) = 4.
$$
If $b_1(X) = 2$, we are in Case (2) with $g(B) = 0$. If $b_1(X) = 4$, then $a$ is generically finite by \cite[Lemma 3.4]{KatsuraUeno}, hence $X$ is an Abelian surface by \autoref{thm: main}, that is, we are in Case (1).

If some fiber of $g$ is contracted by $a$, then $g(B) = \frac{1}{2}b_1(X)$ by \cite[Lemma 3.4]{KatsuraUeno}. Formula \eqref{eq: Noether} implies $b_1(X) > 0$ and Equation \eqref{eq: p2ineq} implies $b_1(X) < 4$, hence $b_1(X) = 2$. Again by \cite[Lemma 3.4]{KatsuraUeno}, $g$ can be identified with the Albanese map of $X$.
By \eqref{eq: nowild}, $g$ admits no wild fibers.
\end{proof}

\begin{proof}[Proof of \autoref{prop: mainintro2}]
    We may assume that $X$ is minimal. By \autoref{lem: chi2}, it is enough to show that $\chi(\mathcal{O}_X) = 0$. Since \[ h^1(X, \mathcal{O}_X) = \dim(T_0\Pic^0_X) \geq \dim(T_0(\Pic^0_X)_{\red}) = \frac{b_1(X)}{2} \] (see e.g. \cite[Lemma 1.6]{Roessler_Schroer_Varieties_with_free_tangent_sheaves} for the last equality), we immediately obtain that $\chi(\mathcal{O}_X) \leq 0$. We then conclude the proof by \autoref{lem: chi1}.
\end{proof}

\section{Proof of \autoref{thm: mainintro}}
In this section, we prove \autoref{thm: mainintro}. This is done by combining the results of the previous section with the results of Katsura--Ueno \cite{KatsuraUeno} and Ferrari \cite{Ferrari} that use Raynaud's theory of jumping numbers for wild fibers.

\begin{proof}[Proof of \autoref{thm: mainintro}]
As always, we assume that $X$ is minimal. By \autoref{lem: chi2}, we have to exclude the case where $b_1(X) = 2$ and $X$ admits a genus $1$ fibration $g \colon X \to B$ with $g(B) \in \{0,1\}$ but is not a (quasi-)bielliptic surface with trivial canonical.

First, assume that $g(B) = 0$. 
Then, by \autoref{lem: chi2}, the Jacobian of $g$ is smooth and by \cite[Lemma 2.3]{Ferrari}, it holds that ${\rm length}(\mathcal{T}) = 2$. In particular, there are at most two wild fibers.
Let $F_1,\hdots,F_r$ be the (underlying reduced schemes of the) multiple fibers of $g$. Let $m_i$ be the multiplicity of $F_i$, $\nu_i$ the order of the normal bundle of $F_i$, and $a_i$ the coefficient of $F_i$ in \eqref{eq: canonical}.
By \cite[Corollary 4.2, Proof of Theorem 5.2, Case (III)]{KatsuraUeno}, one has $p_2(X) \geq 2$ unless $r \leq 2$ and all multiple fibers of $g$ are wild with $a_i < m_i-1$.
By \cite[Lemma 2.4, Corollary 4.2]{KatsuraUeno}, we have the following cases:
\begin{enumerate}
    \item $r = 1, \nu_1 = 1, m_1 = p^{\gamma_1}$ with $\gamma_1 \geq 1$, $a_1 = m_1 - 2$.
    \item $r = 1, \nu_1 = 1, m_1 = p^{\gamma_1}$ with $\gamma_1 \geq 1$, $a_1 = m_1 - 3$.
    \item $r = 1, \nu_1 = 1, m_1 = p^{\gamma_1}$ with $\gamma_1 \geq 1$, $a_1 = m_1 - p - 2$.
    \item $r = 2, a_i = m_i - \nu_i - 1$.
\end{enumerate}
Since $\omega_X \not \cong \mathcal{O}_X$, we may assume $a_1 > 0$ and since the case $p \geq 5$ is done in \cite[Proposition 3.6]{Ferrari} (by an explicit computation similar to below), we only need to treat the case $p \in \{2,3\}$. In any case, from \eqref{eq: canonical}, we have
\begin{equation} \label{eq: ineq}
p_i(X) \geq \left(\sum_{i=1}^r \lfloor \frac{i a_i}{m_i} \rfloor\right) +1.
\end{equation}
Let us go through each case separately:
\begin{enumerate}
\item If $p = 3$, we get
$$
p_3(X) \geq  \lfloor 3 (1 - \frac{2}{3^{\gamma_1}})   \rfloor +1\geq \lfloor 3 (1 - \frac{2}{3})   \rfloor +1 = 2
$$
and if $p = 2$, then $a_1 > 0$ implies $\gamma_1 \geq 2$, so we get
$$
p_2(X) \geq \lfloor 2 (1 - \frac{2}{2^{\gamma_1}})   \rfloor +1\geq \lfloor 2 (1 - \frac{1}{2})   \rfloor +1 = 2.
$$
\item If $p = 3$, then $a_1 > 0$ implies $\gamma_1 \geq 2$, so
$$
p_2(X) \geq \lfloor 2 (1 - \frac{3}{3^{\gamma_1}})   \rfloor +1 \geq \lfloor 2 (1 - \frac{1}{3})   \rfloor +1  = 2.
$$
Similarly, if $p = 2$, then $a_1 > 0$ hence $\gamma_1 \geq 2$ and so
$$
p_4(X) \geq \lfloor 4 (1 - \frac{3}{2^{\gamma_1}})   \rfloor +1 \geq \lfloor 4 (1 - \frac{3}{4})   \rfloor +1 = 2.
$$
\item If $p = 3$, we have $\gamma_1 \geq 2$ and so
$$
p_3(X) \geq \lfloor 3 (1 - \frac{5}{3^{\gamma_1}})   \rfloor +1 \geq \lfloor 3 (1 - \frac{5}{9}) \rfloor +1  = 2.
$$
If $p = 2$, then $a_1 > 0$ actually implies $\gamma_1 \geq 3$ and so
$$
p_2(X) \geq \lfloor 2 (1 - \frac{4}{2^{\gamma_1}})   \rfloor +1 \geq \lfloor 2 (1 - \frac{1}{2})   \rfloor +1 = 2.
$$
\item Here, \eqref{eq: ineq} becomes
$$
p_i(X) \geq \lfloor i (1 - \frac{1}{p^{\gamma_1}} - \frac{1}{m_1}) \rfloor + \lfloor i (1 - \frac{1}{p^{\gamma_2}} - \frac{1}{m_2}) \rfloor  + 1.
$$
If $p = 3$, the minimal positive value of $(1 - \frac{1}{p^{\gamma_1}} - \frac{1}{m_1})$ is attained for $\gamma_1 = 1$, $m_1 = 3$, in which case $p_3(X) \geq 2$.
If $p = 2$, the minimal positive value is attained for $\gamma_1 = 1$ and $m_1 = 4$, in which case $p_4(X) \geq 2$.

 \end{enumerate}

Next, assume that $g(B) = 1$. By \autoref{lem: chi2}, all fibers of $g$ are tame. Let $J(g) \colon J(X) \to B$ be the Jacobian fibration. 
Since $g$ has no wild fibers, $R^1g_* \mathcal{O}_X$ is torsion-free, so $R^1J(g)_* \mathcal{O}_{J(X)} \cong R^1g_* \mathcal{O}_X$ by \cite[Theorem 3.1]{LiuLorenziniRaynaud}. Thus, $p_1(J(X)) = p_1(X) = 1$ by Serre duality. On the other hand, by \eqref{eq: canonical}, the canonical sheaf of $J(X)$ is the pullback of a degree $0$ line bundle from $B$. Thus, $p_1(J(X)) = 1$ implies that it is trivial. By the classification of surfaces of Kodaira dimension $0$, this implies that $J(X)$ is a (quasi-)bielliptic surface with trivial canonical bundle. These surfaces only exist in characteristics $2$ and $3$, so it suffices to show that $p_4(X) \geq 2$.

By \eqref{eq: canonical} and since $g$ has only tame fibers, we have
$$
\omega_X \cong \mathcal{O}_X\left(\sum_{b \in B} (m_b -1)F_b\right),
$$
where $m_b$ is the multiplicity of $F_b$. If at least one $m_b \geq 2$, then $\omega_X^{\otimes 4}$ contains the pullback of a line bundle of degree $2$ from $B$, hence $p_4(X) \geq 2$, contradicting our assumptions. Therefore, $p_4(X) = 1$ is only possible if $\omega_X \cong \mathcal{O}_X$, in which case the classification of surfaces implies that $X$ is either Abelian or (quasi-)bielliptic. 
\end{proof}

\section{Examples}\label{sec: examples}
It is well-known that some assumption on the $p_i(X)$ with $i \geq 2$ is necessary in \autoref{thm: mainintro}, even in characteristic $0$. We thus focus on constructing examples that show the sharpness of the assumptions on the plurigenera in \autoref{thm: mainintro} when $p \in \{2,3,5\}$.

\begin{example} \label{ex: 5}
Here, we show that the bound of \autoref{thm: mainintro} is sharp if $p = 5$.
The proof of \cite[Proposition 3.6]{Ferrari} and \autoref{thm: mainintro} shows that we must look for an example of an elliptic surface $g \colon X \to \mathbb{P}^1$ whose Jacobian $J(g)$ is smooth and such that $g$ has exactly one wild multiple fiber $F$ of multiplicity $5$ such that the underlying reduced fiber has trivial normal bundle and such that the coefficient of $F$ in \eqref{eq: canonical} is $2$. It turns out that the relevant example already appears implicitly in the literature in a series of examples due to Katsura and Ueno \cite[Example 8.1]{KatsuraUeno}:

Assume $p \geq 3$ and consider the smooth curve
$$
C \coloneqq \{x^py - xy^p = z^2\} \subseteq \mathbb{P}(1,1,\frac{p+1}{2}).
$$
of genus $p_a(C) = \frac{1}{2}(p-1)$ and the automorphism $\sigma \colon[x:y:z] \mapsto [x+y:y:z]$ of $C$ of order $p$. Let $E$ be an ordinary elliptic curve and define
$$
X \coloneqq (C \times E)/(\mathbb{Z}/p\mathbb{Z}),
$$
where a generator of $\mathbb{Z}/p\mathbb{Z}$ acts on $C$ through $\sigma$ and on $E$ through translation by a non-trivial $p$-torsion point.
Consider the commutative diagram
$$
\xymatrix{
C \times E \ar^{\tilde{g}}[r] \ar^{\tilde{\pi}}[d] & C \ar^{\pi}[d]\\
X \ar^{g \quad}[r] & \mathbb{P}^1 \cong C/(\mathbb{Z}/p\mathbb{Z})
}
$$
Let $c \in C$ be the ramification point of $\pi$ and let $F$ and $\tilde{F}$ be the (reductions of the) fibers over $\pi(c)$ and $c$, respectively. 
Then, by Riemann--Hurwitz and ramification formulas, we have
\begin{eqnarray*}
K_C &\sim& (p-3)c, \\
K_{C\times E} & \sim & (p-3) \tilde{F}, \\
K_X & \sim & (pl + a -2p)F,
\end{eqnarray*}
where $l$ is the length of the torsion in $R^1g_* \mathcal{O}_X$ and $a$ is the contribution of the multiple fiber $pF$ in the canonical bundle formula. Since $\tilde{\pi}$ is \'etale, we have $\tilde{\pi}^*K_X \sim K_{C \times E}$ and $\tilde{\pi}^*F \sim \tilde{F}$, so
$$
(p-3) \tilde{F} \sim (pl + a - 2p) \tilde{F}.
$$
Since $a < p$, we conclude that $l = 2$ and $a = p-3$. In other words,
$$
K_X \sim (p-3) F.
$$
Now, since $pF$ is a fiber of $g$, we can compute
$$
p_1(X) = \lfloor \frac{p-3}{p} \rfloor +1 = 1, p_2(X) =  \lfloor \frac{2p-6}{p} \rfloor +1.
$$
If $p = 5$, then $2p-6 = 4$, hence $\lfloor \frac{2p-6}{p} \rfloor = 0$ and so $p_2(X) = 1$, as desired.
\end{example}

\begin{example} \label{ex: 3}
Here, we show that the bound of \autoref{thm: mainintro} is sharp if $p = 3$.
The proof of \autoref{thm: mainintro} shows that we must look for an example of an elliptic surface $g \colon X \to B$ with $g(B) = 1$ such that the Jacobian $J(g) \colon J(X) \to B$ is a (quasi-)bielliptic surface with trivial canonical bundle and such that $g$ has exactly one multiple fiber $F$. Moreover, $F$ must be tame and of multiplicity $2$. In the following, we will construct a $1$-dimensional family of elliptic surfaces which specialize to what we need if $p = 3$.

Assume $p \neq 2$. Consider the canonical curve $C$ of genus $4$ and degree $6$ in $\mathbb{P}^3_k$ given by
\begin{eqnarray*}
x_0^2 + x_1^2 + x_2^2 + x_3^2 &=& 0 \\
x_0(x_1^2 + x_2^2 + x_3^2) + \lambda x_1x_2x_3 &=& 0
\end{eqnarray*}
for some $\lambda \in k^{\times}$ with $\lambda^2 \neq -27$. Consider $G = (\mathbb{Z}/2\mathbb{Z})^3 \rtimes \mathbb{Z}/3\mathbb{Z}$ and the action of $G$ on $C$ given by sign changes and cyclic permutation of $x_1,x_2,x_3$. Let $H = (\mathbb{Z}/2\mathbb{Z})^2 \rtimes \mathbb{Z}/3\mathbb{Z} \cong A_4$ be the subgroup where the involutions swap an even number of signs.
The fixed points of the involutions in $V_2 \coloneqq (\mathbb{Z}/2\mathbb{Z})^2$ come in pairs $$p_1^{\pm} = [0:1:\pm \zeta_4: 0], \: \: p_2^{\pm} = [0:1:0:\pm \zeta_4], \: \: p_3^{\pm} = [0:0:1:\pm \zeta_4],$$ so by the Hurwitz formula $B' \coloneqq C/V_2$ is an elliptic curve. The group $\mathbb{Z}/3\mathbb{Z}$ acts on $B'$ without fixed points, so $B \coloneqq C/A_4$ is an elliptic curve as well and the induced morphism $B' \to B$ is \'etale of degree $3$.

% Let $B$ be an elliptic curve over $k$ with $j(B) \neq 0$. Pick a point $O \in B$, a general point $b \in B$, and a non-trivial $3$-torsion point $t \in B$. Consider the three divisors $D_1,D_2,D_3$ determined by $D_i \in |b + (i-1)t|$. Then,
% \begin{eqnarray*}
% D_1 + D_2 &\in & |2b + t| \\
% D_1 + D_3 &\in & |2b + 2t| \\
% D_2 + D_3 &\in & |2b|
% \end{eqnarray*}
% Pick a line bundle $\mathcal{L}$ with $\mathcal{L}^{\otimes 2} \cong \mathcal{O}_B(D_1 + D_2)$ and let $\mathcal{L}'$ be the pullback of $\mathcal{L}$ along translation by $t$, so that $\mathcal{L}'^{\otimes 2} \cong \mathcal{O}_B(D_1 + D_2 +t) \cong \mathcal{O}_B(D_1 + D_3)$. \textcolor{blue}{todo: Instead of this abstract thing, find an explicit equation of a genus $4$ curve with $A_4$-action such that the quotient is an elliptic curve...}
% This yields a Galois cover $\pi \colon C \to B$ with group $V_2 = (\mathbb{Z}/2\mathbb{Z})^2$ with $3$ branch points $D_1,D_2,D_3$ and with $2$ ramification points $D_i^{\pm}$ of index $2$ over each $D_i$. Translation by $t$ lifts to an automorphism of $C$, resulting in an $A_4$-action on $C$ with quotient the elliptic curve $B' \cong B/(\mathbb{Z}/3\mathbb{Z})$.

Now, take an elliptic curve $E$ with $j(E) = 0$. Let $A_4$ act on $C$ via the above action and on $E$ so that $V_2$ acts as translation by the full $2$-torsion of $E$ and the $\mathbb{Z}/3\mathbb{Z}$-part is an automorphism of $E$ fixing the origin. Note that the resulting $A_4$-action on $C \times E$ is free. Consider the quotients $X \coloneq (C \times E)/A_4$ and $X' \coloneqq (C \times E)/V_2$ and the genus $1$ fibrations $g \colon X \to B$ and $g' \colon X' \to B'$.

A general fiber of $g$ splits into twelve fibers of the projection under the cover $C \times E \to X$. The six fibers $F_1^{\pm},F_2^{\pm},F_3^{\pm}$ over the points $p_i^{\pm} \in C$ get identified pairwise to three fibers $F'_1,F'_2,F'_3$ over points $q'_1,q'_2,q'_3 \in B'$ and these three points map to a single point $q \in B$. Thus, the $F'_i$ are double fibers of $g'$ and the fiber $F$ of $g$ over $q$ is a double fiber as well. In particular, $F$ is tame since $p \neq 2$. Note that $\chi(\mathcal{O}_X) = 0$, since $X$ is an \'etale quotient of $C \times E$. By the canonical bundle formula, we thus have
$$
\omega_X \cong g^* \mathcal{L}^{-1} \otimes  \mathcal{O}_X(F),
$$
where $\mathcal{L}^{-1}$ is a line bundle of degree $0$ on $B$.

It remains to compute $\mathcal{L}^{-1}$. The pullback of $\mathcal{L}^{-1}$ to $C \times E$ can be computed through the canonical bundle formula and the composition $C \times E \to X \to B$. It is $(\omega_C \boxtimes \omega_E)(-\sum_{i=1}^3 F_i^{\pm})$, which is the same as the pullback of $\omega_C(-\sum_{i=1}^3 p_i^{\pm})$ along the projection $C \times E \to C$. But now note that the divisor $\sum_{i=1}^3 p_i^{\pm}$ is the intersection of $C$ with the hyperplane $V_+(x_0)$, hence it is a canonical divisor. The pullback of $\mathcal{L}^{-1}$ to $C \times E$ is therefore trivial. Since $\Pic(C) \to \Pic(C \times E)$ is injective, we deduce that the pullback of $\mathcal{L}$ to $C$ is trivial. 

Note that the pullback map $\Pic(B') \to \Pic(C)$ is injective: if not, then by \cite[Proposition 0.2.14]{EnriquesI} we know that $C \to B'$ factors through an étale cover of $B'$, which is impossible since every element of $V_2$ has fixed points in $C$. Thus, $\mathcal{L}$ actually lies in $\ker(\Pic^0(B) \to \Pic^0(B')) \cong \mu_3$. In particular, we deduce that $\mathcal{L}$ is trivial if $p = 3$. In this case, then $\omega_X \cong \mathcal{O}_X(F)$, and hence $p_1(X) = p_2(X) = p_3(X) = 1$ while $p_4(X) \geq 2$. As an aside, note that \autoref{cor: mainintro} implies that $\mathcal{L}$ is not trivial if $p \neq 3$.

%We then obtain that the map $C \to B$ factors through the cyclic cover $B'' \to B$ associated to $\mathcal{L}$ (see e.g. \cite[Proposition 0.2.14]{EnriquesI}). Since $C \to B$ is separable, the order of $\mathcal{L}$ must be prime to $p$.

%Assume that $\mathcal{L}$ is non-trivial. Then, $B'' \to B$ is \'etale. Since every involution in $A_4$ has fixed points on $C$, the only way to factor $C \to B'$ through an \'etale morphism is via $B'' = B'$. Then, $B'' \to B$ has degree $3$, hence $p \neq 3$.

%In particular, if $p = 3$, then $\mathcal{L}$ must be trivial. Then, $\omega_X \cong \mathcal{O}_X(F)$, hence $p_1(X) = p_2(X) = p_3(X) = 1$ while $p_4(X) \geq 2$. As an aside, note that \autoref{cor: mainintro} implies that $\mathcal{L}$ is not trivial if $p \neq 3$.
\end{example}

\begin{example} \label{ex: 2}
In characteristic $2$, there are several ways in which a minimal surface which is not Abelian or (quasi-)bielliptic but satisfies $h^1(X,\mathcal{O}_X) = 2$ and $p_1(X) = p_2(X) = p_3(X) = 1$ can arise. Among these, the simplest one seems to be the following characteristic $2$ analog of \autoref{ex: 3}.

Assume that $p = 2$, and consider the curve $C$ of arithmetic genus $3$ in $\mathbb{P}(1,1,1,2)$ given by
\begin{eqnarray*}
x_3^2 + x_1x_2 x_3 + x_1^4  + x_2^4 &= & 0\\
x_0^2 + x_1x_2 &=& 0.
\end{eqnarray*}
Then $C$ has two cusps, at $[0:0:1:1]$ and $[0:1:0:1]$. Consider the length $8$ dihedral group scheme $G = \mu_4 \rtimes \mathbb{Z}/2\mathbb{Z}$, where $\mathbb{Z}/2\mathbb{Z}$ acts on $\mu_4$ as inversion. There is a $\mu_4$-action on $C$ given by
$$
\tau_a \colon [x_0:x_1:x_2:x_3] \mapsto [x_0 :ax_1:a^{-1}x_2:x_3]
$$
with $a^4 = 1$ and a $\mathbb{Z}/2\mathbb{Z}$-action given by
$$
\sigma \colon [x_0:x_1:x_2:x_3] \mapsto [x_0:x_2:x_1:x_3+x_1x_2]
$$
which together define a $G$-action since $\sigma \circ \tau_a \circ \sigma = \tau_{a^{-1}}$. 

Note that $\tau_a^2 \colon [x_0:x_1:x_2:x_3] \mapsto [x_0:a^2x_1:a^2x_2:x_3] = [a^2x_0:x_1:x_2:x_3]$, so $C/\mu_2$ is the smooth genus $1$ curve $B \subseteq \mathbb{P}(1,1,2)$ given by
$$
x_3^2 + x_1x_2x_3 + x_1^4 + x_2^4.
$$
Neither $\mu_4$ nor $\mathbb{Z}/2\mathbb{Z}$ have fixed points on $C$, so the locus of points on $C$ with non-trivial stabilizer is the fixed locus of $\mu_2$, which is $C^{\mu_2} = V(x_0,x_1x_2,x_3^2 +x_1^4 +x_2^4)$. Note that $(C^{\mu_2})_{\red}$ is the union of the two cusps of $C$. The action of $G/\mu_2 = \mu_2 \times \mathbb{Z}/2\mathbb{Z}$ on $B$ can be checked to be free. It is the translation action of the $2$-torsion subscheme of ${\rm Aut}^0_B$.

Now, even though $C$ is singular, the quotient stack $[C/\mu_4]$ is smooth. This is clear away from the singular locus of $C$ because the chart $C \to [C/G]$ is smooth there. For the rest, by symmetry, it suffices to show that on the affine chart $x_1 = 1$, where $C$ is given by $C_2 = \Spec k[x_0,x_3]/(x_3^2 + x_0^2x_3 + 1 + x_0^8)$ with the $\mu_4$-action $(x_0,x_3) \mapsto (ax_0,a^2x_3)$ the quotient stack $[C_2/\mu_4]$ is smooth. By \cite[Proposition 4.17]{Brion}, this is equivalent to the smoothness of the contracted product 
$$
C_2 \times^{\mu_4} \mathbb{G}_m = \Spec (k[x_0,x_3,b,b^{-1}]/(x_3^2 + x_0^2x_3 + 1 + x_0^8))^{\mu_4},
$$
where $\mu_4$ acts as $(x_0,x_3,b) \mapsto (ax_0,a^2x_3,ab)$. Now, this invariant ring contains the smooth subalgebra
\begin{eqnarray*}
R &=& k[x_0b^{-1},x_3b^{-2},b^4,b^{-4}]/(x_3^2b^{-8} + x_0^2x_3b^{-8} + b^{-8} + x_0^8b^{-8})  \\ &\cong & k[x,y,t,t^{-1}]/(y^2t + x^2yt + t^2 + x^8),
\end{eqnarray*}
and $R[t^{\frac{1}{4}}] \cong k[x_0,x_3,b,b^{-1}]/(x_3^2 + x_0^2x_3 + 1 + x_0^8)$. Comparing degrees, we see that $C_2 \times^{\mu_4} \mathbb{G}_m$ is finite and birational over ${\rm Spec}(R)$, hence the two are isomorphic. Thus, $[C/\mu_4$] is smooth.

Let $E$ be another ordinary elliptic curve and consider the diagonal quotient $X \coloneqq (C \times E)/G$, where $G^0 = \mu_4$ acts on $E$ as translations and $\mathbb{Z}/2\mathbb{Z}$ acts as inversion.
By \cite[Proposition 4.17 and Remark 4.19]{Brion}, the smoothness of $[C/\mu_4]$ implies that $X' \coloneqq (C \times E)/\mu_4$ is smooth, and since $G$ acts freely on $C \times E$ this implies the smoothness of $X$.

We thus have the following commutative diagram
$$
\xymatrix{
C \times E \ar^{\pi''}[r] \ar^g[d] & X'' \ar[r]  \ar^{f''}[d] & X' \ar[r]  \ar^{f'}[d]& X  \ar^f[d] \\ 
C \ar[r] & B'' = C/\mu_2 \ar[r] & B' = C/\mu_4 \ar[r] & B = C/G
}
$$
where the top horizontal arrows are torsors under length $2$ group schemes and the vertical arrows are genus $1$ fibrations with general fiber $E$. The surfaces $C \times E$ and $X''$ are non-normal, but $X'$ and $X$ are smooth. We compute the invariants of the elliptic fibration $f'$ using \eqref{eq: canonical} and the $\mu_4$-torsor $\pi \colon C \times E \to X'$. Away from the two points $c_1,c_2 \in C^{\mu_2}$, $C \times E \to C$ is the pullback of $f'$ along $h \colon C \to B'$, so $f'$ has no multiple fibers there. Comparing pullbacks for the other two fibers, we see that $f'$ has two double fibers $2F_1'$ and $2F_2'$. We have
\begin{eqnarray*}
\pi^* \omega_{X'} &\cong& \omega_{C \times E} \cong g^* \omega_C \\
\pi^*\omega_{X'} & \cong & g^*(h^* \mathcal{L}'^{-1}) \otimes_{\mathcal{O}_{C \times E}} \mathcal{O}_{C \times E}(2a_1'g^{-1}(c_1) + 2a_2'g^{-1}(c_2)),
\end{eqnarray*}
where $\mathcal{L'}^{-1}$ is an invertible sheaf of degree equal to the length of the torsion in $R^1f'_* \mathcal{O}_{X'}$ and $a_i'$ is the coefficient of $F_i'$ in \eqref{eq: canonical}. Note that $\omega_{C}$ has degree $4$ and ${\rm deg}(h^*\mathcal{L}'^{-1}) = 4 \deg(\mathcal{L}'^{-1})$. Since the two multiple fibers of $f'$ are permuted by the $\bZ/2\bZ$--action, it follows that they are either both tame or wild. In the latter case, then it would follow that $\deg(\mathcal{L}'^{-1}) \geq 2$, and therefore give that $4 = \deg(\omega_C) \geq 4\deg(\mathcal{L}'^{-1}) \geq 8$, a contradiction. Hence, both fibers are tame, and $a_1' = a_2' = 1$.
% From our explicit description of $C$, we also know that $2c_1 + 2c_2$ is the zero locus of a section of $\omega_C \cong \mathcal{O}_{\mathbb{P}(1,1,1,2)}(1)|_C$, so applying the projection formula to the above descriptions of $\pi^* \omega_{X'}$ we deduce that $h^* \mathcal{L}'^{-1} \cong \mathcal{O}_C$. 
Thus,
$$
\omega_{X'} \cong f'^* \mathcal{L}'^{-1} \otimes_{\mathcal{O}_{X'}} \mathcal{O}_{X'}(F_1' + F_2')
$$
with $\deg(\mathcal{L}') = 0$.

Next, observe that, by adjunction, the embedding $C \to \mathbb{P}(1,1,1,2)$ is induced by the dualizing sheaf $\omega_C$. As $\pi^* \omega_{X'} \cong g^* \omega_C$, the $\mu_4$-invariant section $x_0 |_C \in H^0(C,\omega_C)$ yields a non-zero global section of $\omega_{X'}$ and hence a non-zero global section of
$$
f'_* \omega_{X'} \cong \mathcal{L'}^{-1}.
$$
Thus, $\mathcal{L}' \cong \mathcal{O}_{B'}$.

Now, using that the covering involution of $X' \to X$ swaps $F_1'$ and $F_2'$ and letting $F$ be their image, we get
$$
\omega_X \cong f^* \mathcal{L}^{-1} \otimes_{\mathcal{O}_X} \mathcal{O}_X(F),
$$
where $\mathcal{L}$ pulls back to $\mathcal{O}_{B'}$ under $B' \to B$. Since $B' \to B$ is \'etale of degree $2$ and $p = 2$, the pullback map $\Pic(B) \to \Pic(B')$ is injective, hence $\mathcal{L} \cong \mathcal{O}_B$. Thus, $\omega_X$ is not trivial while $p_1(X) = p_2(X) = p_3(X) = 1$ and $p_4(X) \geq 2$, as desired.
\end{example}

 \begin{remark}
 If $p \neq 2$, the curve $C$ in \autoref{ex: 2} is smooth with a $D_8$-action and the analogous construction yields an elliptic surface $f \colon X \to C/D_8$ with a unique tame double fiber and $R^1f_*\mathcal{O}_X$ is a non-trivial $2$-torsion line bundle, hence $p_1(X) = 0$.
 \end{remark}

\iffalse
\section{The case $h^1(\cO_S) = 2$ and $b_1(S) = 2$}

I'm just recalling some arguments towards where I remember we were. Throughout, $S$ denotes a minimal surface of Kodaira dimension $1$ such that $P_1(S) = P_2(S) = 1$, $h^1(\cO_S) = 2$ and $b_1(S) = 2$. We denote by $f \colon S \to C$ its Iitaka fibration, and by $a \colon S \to E$ its Albanese morphism (so $E$ is an elliptic curve).

Here's a thought I had: I believe that automatically, 

\begin{lemma}
    There are two possibilities:
    \begin{enumerate}
        \item $C = \bP^1$, and $f$ is not quasi-elliptic, but wild with $\length(\cT) = 2$. In this case, all fibers of $f$ must dominate $E$ (which I think means that all their reductions are smooth elliptic curves, right? I really feel like we can push further in this case).
        \item $f = a$ and the fibration admits no wild fiber.
    \end{enumerate}
\end{lemma}

\begin{proof}
    If $C = \bP^1$ and the fibration was quasi-elliptic, then this would imply that the image of $S \to E$ is trivial.
\end{proof}
\fi

    \bibliographystyle{amsalpha}
	\bibliography{refs}
	
\end{document}